\documentclass[letterpaper, twocolumn, 10 pt, conference]{ieeeconf}  

\IEEEoverridecommandlockouts                              

\usepackage{graphics} 
\usepackage{epsfig} 
\usepackage{mathptmx} 
\usepackage{times} 
\usepackage{amsmath} 
\usepackage{amssymb}  
\usepackage{graphicx}
\usepackage{relsize}
\DeclareMathAlphabet{\mathcal}{OMS}{cmsy}{m}{n}
\SetMathAlphabet{\mathcal}{bold}{OMS}{cmsy}{b}{n}

\title{\LARGE \bf
Grid-Forming ($\lambda-\omega$) Virtual Oscillator Control in Converter-Based Power Systems
}

\author{Taouba Jouini$^{1}$, Emma Tegling$^{1}$ and Zhiyong Sun$^{2}$    
\thanks{*This work has received funding from the European Research Council (ERC) under the European Union's Horizon 2020 research  and innovation program (grant agreement no: 834142) and from the Swedish Research Council under grant 2019-00691.}%
\thanks{ $^{1}$ T. Jouini and E. Tegling are with the Department of Automatic Control, Lund University, Sweden. $^{2}$ Z. Sun is with Department of Electrical Engineering,  Eindhoven University of Technology, the Netherlands.
        E-mails: \tt\small \{taouba.jouini, emma.tegling\}@control.lth.se, z.sun@tue.nl.}}%

\graphicspath{{./fig/}}
\usepackage[usenames,dvipsnames]{color}
\usepackage{epstopdf}

\usepackage{amsmath,amsthm,amssymb,mathrsfs,dsfont}
\usepackage{amsfonts}
\usepackage{siunitx}
\usepackage[colorinlistoftodos,prependcaption,textsize=tiny]{todonotes}
\usepackage{tikz}
\usetikzlibrary{matrix}

\usepackage{cite}

\newcommand\oprocendsymbol{\hbox{$\blacksquare$}}

\newcommand\oprocend{\relax\ifmmode\else\unskip\hfill\fi\oprocendsymbol}

\newcommand{\real}[0]{\mathbb R}

\providecommand{\norm}[1]{\lVert#1\rVert}

\newtheorem{theorem}{Theorem}[section]
\newtheorem{lemma}[theorem]{Lemma}

\newtheorem{proposition}[theorem]{Proposition}

\newtheorem{remark}[]{Remark}
\newtheorem{assumption}[]{Assumption}

\newtheorem{condition}[theorem]{Condition}

\usepackage{verbatim}

\usepackage[normalem]{ulem}
\newcommand\rout{\bgroup\markoverwith{\textcolor{red}{/}}\ULon} 

\usepackage[prependcaption,colorinlistoftodos]{todonotes}

\begin{document}

\maketitle
\thispagestyle{empty}
\pagestyle{empty}

\begin{abstract}
Inspired by the kinetics of wave phenomena in reaction-diffusion models of biological systems, we propose a novel grid-forming control strategy for control of three-phase DC/AC converters in power systems. The ($\lambda-\omega$) virtual oscillator control or ($\lambda-\omega$) VOC is a natural increment on ideas from virtual oscillator dynamics rotating at a fixed nominal frequency to adaptive, angle-based frequency function. We study a network of identical three-phase DC/AC converters interconnected via $\Pi$ transmission lines. For this, we prove almost global asymptotic stability for a reduced (time-scale separated) version of the model, associated to a well-defined set of controller gains and system parameters. Additionally, we link the ($\lambda-\omega$) VOC to well-studied controllers in the literature, e.g. droop control. Finally, we validate our results on an example three DC/AC converter network.
\end{abstract}


\section{INTRODUCTION}
Power generation has always been a major source of air pollution and much effort has been devoted to developing cleaner generation technologies. However, the relatively recent concerns about global warming and sustainability have started to change the way power systems operate and expand and motivated a major shift to renewables with converter-interfaced generation \cite{machowski1997power}. In the presence of large amounts of converter-interfaced generation, the modeling and design of controllers in power systems needs to be re-thought from new perspectives \cite{paolone2020fundamentals}.
\paragraph{Literature} In essence, the control design for power converters has been driven by two directions of research. On the one hand, the design of oscillators on the circle (or torus), where AC signals are typically assumed to be in quasi-stationary steady state and hence, can be described by their phase angles. The change of the frequency is dictated by the power fluctuations in the grid. This gives rise to droop control \cite{SIMPSONPORCO20132603} and variations thereof \cite{bevrani2014virtual}, which are designed to emulate the dynamics of synchronous machines.
Another controller that well translates the advantageous properties of synchronous machines is the matching controller \cite{arghir2018grid} that makes use of available measurements of the voltage deviations of the DC capacitor to {\em define} the frequency for the electrical grid. 

On the other hand, converter control design based on oscillators in rectangular coordinates has also gained attention and is founded on the developed theory of nonlinear oscillations~\cite{fradkov1998introduction}; The Van der Pol oscillator, or more generally Liénard oscillator, gives rise to virtual oscillator control (VOC) in \cite{J14}, Hopf-like bifurcations inspire dispatchable virtual oscillator (d-VOC) in \cite{colombino2019global} and alike \cite{li2020advanced}. Once expressed in polar coordinates after proper time-averaging, virtual oscillator control possesses droop properties that guarantee its compatibility with oscillators on the circle \cite{sinha2015virtual,colombino2019global} and synchronous machines.

\paragraph{Contributions}
 Our contributions in this work are as follows: first, we formulate a control design problem for identical three-phase DC/AC converters whose dynamics are represented by a controllable current source and interconnected via dynamical transmission lines. Motivated by ($\lambda-\omega$) oscillations in biological systems that represent the kinetics of wave phenomena in reaction-diffusion models \cite{murray2007mathematical}, we propose a grid-forming ($\lambda-\omega$) virtual oscillator control ($\lambda-\omega$) VOC. In particular, given prescribed amplitude and phase angle setpoints, we provide more understanding for the use of phase angles in the control design of DC/AC converters and how it can influence transient stability in converter-based power systems. In {particular, we define the system frequency as the {\em angle difference} relative to desired angle set-point rotating at the frequency of interest}. Second, we build upon previously studied virtual oscillator control strategies. Instead of generating sinusoidal waves with a fixed frequency, we introduce an adaptive frequency function based on deviation of phase angles relative to their desired values. By deploying tools from singular perturbation theory, we provide a detailed stability analysis for a reduced model of the closed-loop dynamics. We prove almost global asymptotic stability of the desired steady state set for an explicitly characterized range of system parameters and control gains. Third, we study the properties of the proposed controller and link them to well-studied control strategies, e.g. droop control. Fourth, we corroborate our results through numerical simulations of identical three-phase DC/AC converters in closed-loop with ($\lambda-\omega$) VOC.
 
 Our paper unfurls as follows: Section \ref{sec: setup} presents the setup of the power system model and the overall control objectives. Section \ref{sec: control-design} introduces our grid-forming ($\lambda-\omega$) VOC and links it to other control design strategies. Section \ref{sec: stability} studies the global stability of the desired set of a reduced-order model and identifies parameters and control gains that guarantee stability.
Finally, Section \ref{sec: sims} exemplifies our theory via simulations of a three converter system.
\paragraph{Notation}
We define an undirected graph $G=(\mathcal{V}, \mathcal{E})$, where $\mathcal{V}$ is the set of nodes with $\vert\mathcal{V}\vert =n$ and $\mathcal{E}\subseteq \mathcal{V}\times \mathcal{V}$ is the set of  edges with $\vert \mathcal{E}\vert=m$. The network topology is specified by the incidence matrix~$\mathcal{B}\in\mathbb{R}^{n\times m}$. We denote the $2\times2$ identity matrix by $I_2=\left[\begin{smallmatrix}
1& 0\\ 0& 1
\end{smallmatrix}\right]$ and the identity matrix of appropriate dimensions $p\in\mathbb{N}$ by $\mathcal{I}$. We define $\mathcal{J}= \mathcal{I} \otimes J_2$ with $J_2=\left[\begin{smallmatrix}
0& -1\\ 1& 0
\end{smallmatrix}\right]$. Let $\mathbb{S}^1$ define the unit circle. For $\gamma\in\mathbb{S}^1$, we define the transformation matrix by $R(\gamma)=\left[\begin{smallmatrix}
\cos(\gamma) & \sin(\gamma) \\ -\sin(\gamma) &\cos(\gamma) 
\end{smallmatrix}\right]$,
and $\mathcal{R}(\gamma)= \mathcal{I}\otimes R(\gamma)$.
Let $\mathds{1}_n$ be $n$-dimensional vector with all entries being one.
Given a set $\mathcal{A}$, we denote by $\norm{x}_\mathcal{A}=\inf_{a\in\mathcal{A}} d(x,a)$ the distance from a vector $x\in\mathbb{R}^n$ to the set~$\mathcal{A}$, where $d(\cdot,\cdot)$ is a distance metric. Given a dynamical system $\dot x= f(x)$, we denote by $x(t,x_0)$ its  trajectory at times $t>~0$ starting at $x(0)=x_0$. Given a matrix $M$, we denote by $\norm{M}$ its 2-norm, and given a symmetric matrix $A$, we denote by $\lambda_{\min}(A)$ and $\lambda_2(A)$ its smallest and its second-smallest eigenvalue, respectively.
\section{Control objectives}
\label{sec: setup}
\subsection{Network model}
\begin{figure}[ht!]
    \centering
    \includegraphics[scale=0.25, trim=2cm 0cm 0cm 2cm]{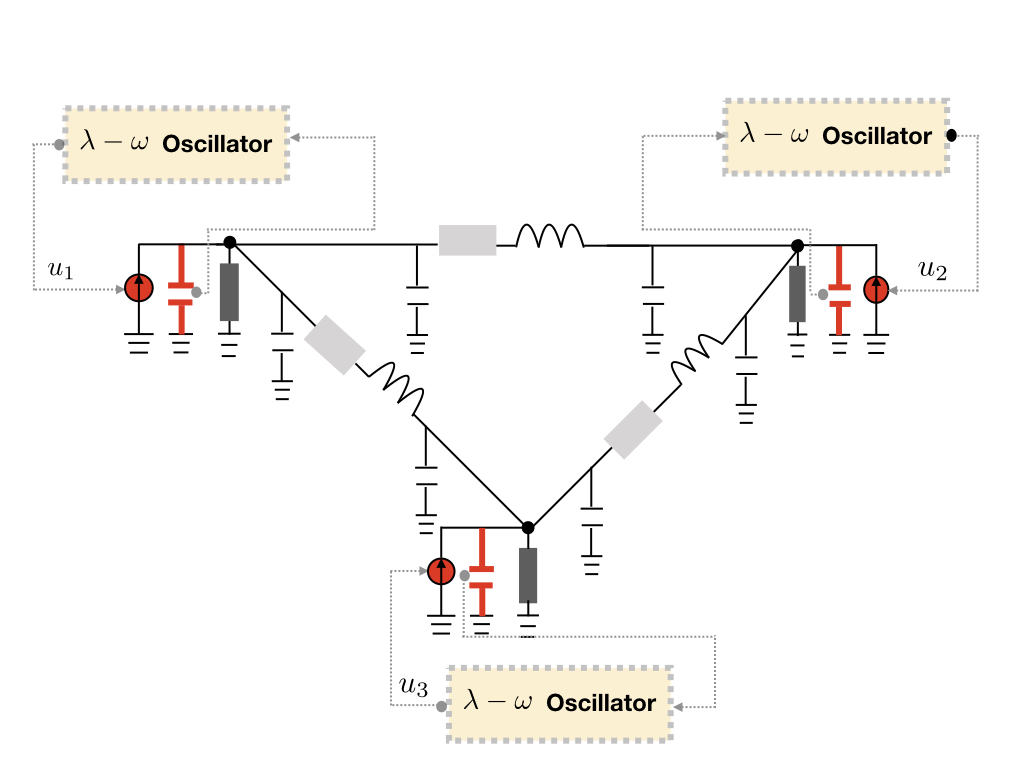}
    \caption{Power system setup under consideration}
    \label{fig:setup}
\end{figure}
Consider a power system model as shown in Figure \ref{fig:setup} consisting of averaged and balanced three-phase DC/AC converters  and connected via $\Pi$ transmission lines. Each line consists of a resistance $R_O>0$ and an inductance $L_O>0$ connected to a shunt capacitor with $C_O>0$. 
Each DC/AC converter is represented by a controllable current source set in parallel with an output capacitor. This is lumped with line capacitance into $C>0$ and connected to a resistive load with conductance $G>0$. The dynamics of the power system model can be described in the $\alpha\beta$-frame as follows, 
\begin{align}
\label{eq: pow-sys}
   C \,\dot v_c &= -G\, v_c- \mathbf{B} \, i_{O}+  u,\nonumber \\
     L_{O}\,\dot i_{O}&= -R_{O}\, i_{O}+ \mathbf{B}^\top\, v_c,
\end{align}
where $i_{O}=[i^\top_{O,1}, \dots,i^\top_{O,m}]^\top\in\real^{2m}$ is the line current and $v_c=[v^\top_{c,1}, \dots , v^\top_{c,n}]^\top\in\real^{2n}$ is the output capacitor voltage. We denote by $\mathbf{B}=I_2\otimes \mathcal{B}$ the extended incidence matrix.
Let $u=[u^\top_{1}, \dots , u^\top_{n}]^\top\in\real^{2n}$ be a controllable current source that represents our main control input to be designed in the sequel. For an input $u_k\in\real^2$, we denote its polar coordinates by $[r_k,\theta_k]^\top$, where  $u_{\alpha,k}=r_k \cos(\theta_k),u_{\beta,k}=r_k \sin(\theta_k)$, as well as $r=[r_1,\dots, r_n]^\top$ and $\theta=[\theta_1,\dots,\theta_k]^\top$.

\subsection{Control objectives}
We increment \eqref{eq: pow-sys} as follows, 
\begin{align}
\label{eq: objectives}
  u&=\mathcal{F}\,(u(s)\,, \hat y(s),
  0\leq s\leq t),\nonumber\\ 
  C \,\dot v_c &= -G\, v_c- \mathbf{B} \, i_{O}+  u,\nonumber\\
L_{O}\,\dot i_{O}&= -R_{O}\, i_{O} + \mathbf{B}^\top\, v_c,
 \end{align}
where $\mathcal{F}$ is an operator and $\hat y$ is the output measurement that will be specified in the sequel.
 
Let $N$ denote the dimension of the closed-loop  system \eqref{eq: objectives}, $\omega^*$ a desired steady state frequency and define $x=[u^\top, v^\top_c, i^\top_O]^\top$.
Given prescribed amplitudes $r^*=[r^*_1,\dots, r^*_n]^\top$ and phase angles $\theta^*=[\theta^*_{1},\dots, \theta^*_{n}]^\top$, we use the measurement $\hat y$ to design the control input $u$. This is achieved by designing the operator $\mathcal{F}$ as a dynamical oscillator $\dot u =f(u,\hat y)$ with vector field ${f}(u,\hat y)=[f_1( u_1,\hat y_1),\dots, f_n(u_n,\hat y_n)]^\top$ that induces sinusoidal waves at a desired amplitude $r^*$ and a desired angle $\theta^*=\omega^* t+\theta_0^*,\,\{\theta_0^*\in\mathbb{S}^1\}$.
In particular, we design the input $u$ to stabilize \eqref{eq: objectives} at a desired synchronous steady state set described by, 
 $$\mathcal{S}=\left\{ x\in\real^N,\, x=x^*\right\},$$ whereby,
    $$\mathcal{S}_u=\left\{u\in\real^{2n}, \,[r^\top,\theta^\top]^\top= [r^{*\top},\theta^{*\top}]^\top, \dot\theta^* =\omega^*\;\mathds{1}_n\right\}.$$ 
Note that the steady state values of $x^*$ are induced by the desired magnitude and phase angle of $[r^{*\top},\theta^{*\top}]^\top$. The sets $\mathcal{S}$ and $\mathcal{S}_u$ are frequency synchronous and $2\pi$-periodic. 

\section{Grid-forming ($\lambda-\omega$) VOC}
\label{sec: control-design}
\subsection{Control design}
Our control goal is to sustain AC oscillations at specified amplitude $r^*$ and phase angle $\theta^*$. For this, we propose the nonlinear dynamics inspired by {$(\lambda-\omega)$ systems} representing the kinetics of wave phenomena in reaction-diffusion models of biological systems \cite{murray2007mathematical},  
\begin{align*}
    f_k:=\left[\begin{smallmatrix}
        -\lambda_k(r_k,r^*_{k}) & \widetilde\omega_k(\theta_k, \theta^*_{k}) \\ -\widetilde\omega_k(\theta_k, \theta^*_{k}) & -\lambda_k(r_k,r^*_{k})
    \end{smallmatrix} \right] u_k+ {\widetilde h}_k(v_{c,k}).
\end{align*}
Here  $u_k$ is the input and $u^*_k$ is a given reference in polar coordinates  $[r^{*\top}, \theta^{*\top}]^\top$ and we have set $\hat y_k=v_{c,k}$. The function ${\widetilde h}_k(\cdot)$ depends explicitly on the local output voltage measurement $v_{c,k}$, whose reference $v^*_{c,k}$ is assumed to be available. We introduce the following amplitude function, 
 \begin{align}
 \label{eq: amp-fcn}
\lambda_k(r_k, r^*_k)=\gamma_k\,(\frac{r_k}{r^*_k}-1),     
 \end{align} 
 as well as the frequency function, 
 \begin{align}
  \label{eq: angl-fcn} 
 \widetilde\omega_k(\theta_k, \theta^*_k)=\alpha_k(\theta_k-\theta^*_k) -\omega^*,\, \alpha_k>0,\end{align}
where $\gamma_k$ and $\alpha_k$ are positive control gains and $\theta\in[0,2\pi)$. 
Note that, the frequency function \eqref{eq: angl-fcn} is a novelty in our work in comparison to other suggestions of virtual oscillator control. For a discussion, see Section \ref{subsec: interpre}. 

If ${\widetilde h}_k(v_{c,k})=0$, then the ($\lambda-\omega$) VOC dynamics can be expressed compactly in polar coordinates, as follows
\begin{align}
 \begin{bmatrix}
 \dot r_{k} \\ \dot \theta_{k}
 \end{bmatrix}= \begin{bmatrix}
 -r_{k} \,\lambda_k(r_{k}, r^*_{k}) \\ -\widetilde\omega_k(\theta_k, \theta^*_{k})
 \end{bmatrix}.
 \end{align}

Next, we design $${\widetilde h}_k(v_{c,k},v^*_{c,k})= \Pi_k v_{c,k},\;\; \Pi_k=I_2-\frac{v^*_{c,k} v^{*\top}_{c,k}}{v^{*\top}_{c,k}v^{*}_{c,k}}.$$

Note that, for $r_k=r_{k}^*$ and $\theta_k=\theta_k^*$, the ($\lambda-\omega$) VOC is a harmonic oscillator rotating at the desired frequency $\omega^*$.

By defining, 
\begin{align*}
\widetilde K\,(\theta,\theta^*)&= \mathcal{I}\otimes (J_2\, \widetilde\omega_k(\theta_k,\theta^*_{k})),\\
\Lambda\,(r,r^*)&= \mathcal{I}\otimes \left(I_2\,\lambda_k(r_k, r^*_{k})\right),\\ 
\Pi_c&=\mathrm{diag}(\Pi_1,\dots,\Pi_n ),
\end{align*}
the closed-loop system becomes,
\begin{subequations}
\label{eq: lambda-omega-voc}
\begin{align}
\dot u&=-\bigg( \Lambda\,(r,r^*)+\widetilde K\,(\theta,\theta^*)\bigg)\, u- \,\Pi_c v_c, \label{eq: input} \\
C \,\dot v_c &= -G\, v_c- \mathbf{B}\,i_{O}+ u,\\
L_{O}\,\dot i_{O}&= -R_{O}\, \, i_{O}+ \mathbf{B}^\top\, v_c.
\end{align}
\end{subequations}

A summarizing diagram of the overall closed-loop system is presented in Figure \ref{fig: diag}. 
\begin{figure}[h!]
     \centering
     \includegraphics[scale=0.35, trim=4cm 8cm 0cm 0cm, clip]{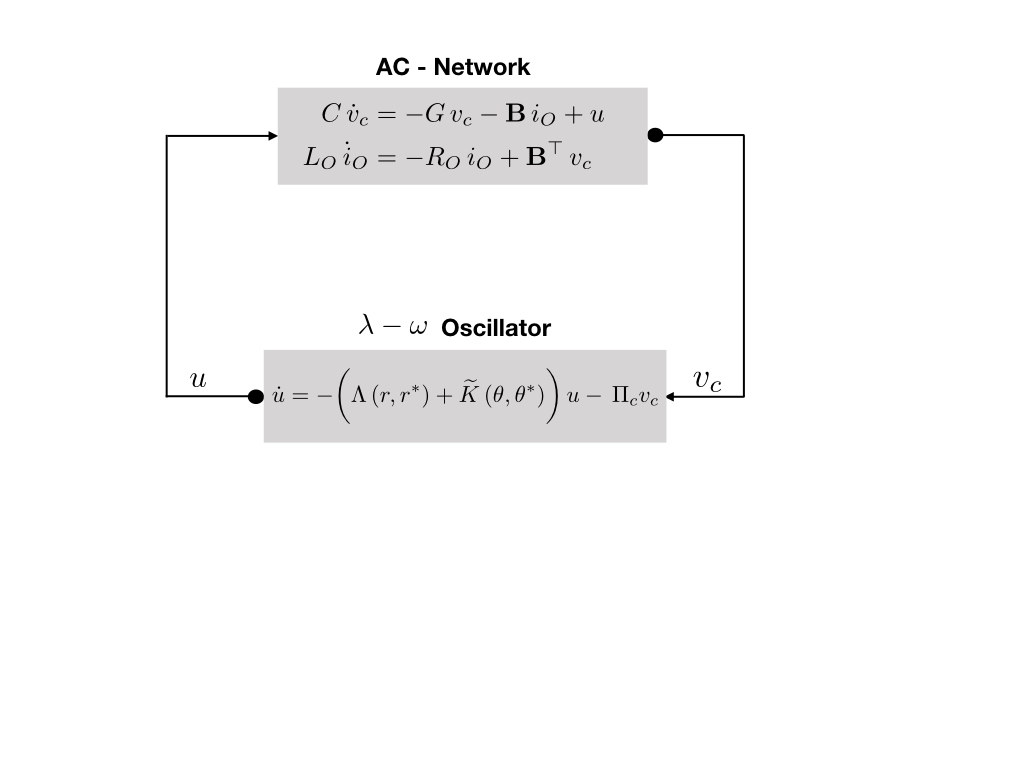}
     \caption{Summarizing diagram of $(\lambda-\omega)$ VOC.}
     \label{fig: diag}
 \end{figure}

\subsection{Interpretation of grid-forming ($\lambda-\omega$) VOC}
\label{subsec: interpre}

By design, the ($\lambda-\omega$) virtual oscillator in \eqref{eq: lambda-omega-voc} is a grid-forming converter in the sense of \cite{debry2019characterization}. 

\paragraph{Droop control}
Droop control is a well-studied control strategy both in theory and practice. For this reason, it is regarded as a compass for evaluating the performance of novel control strategies as well as their compatibility with other controllers (with droop behavior) and synchronous machines. \begin{proposition}
\label{prop: droop}
Consider the grid-forming ($\lambda-\omega$) VOC in \eqref{eq: lambda-omega-voc}. Define 
\begin{align*}P_k&= u_{k}^\top v_{c,k},~~ P_{ref,k}= u_k^\top \frac{v^{*}_{c,k} v^{*\top}_{c,k}}{v^{*\top}_{c,k} v^{*}_{c,k}}\,v_{c,k},\\
Q_k&= u_{k}^\top  {J}_2  v_{c,k},~~ Q_{ref,k}= u_k^\top J_2 \frac{v^{*}_{c,k} v^{*\top}_{c,k}}{v^{*\top}_{c,k} v^{*}_{c,k}}\, v_{c,k}.\end{align*}
Then, the ($\lambda-\omega$) VOC has a droop-like behavior given by the following equations,
\begin{subequations}
\begin{align}
\label{eq: amp-droop}
\dot r_k&= -\gamma_k\, r_k\left(\frac{r_k}{r^*_{k}}-1\right)+ \frac{-P_k+ P_{ref,k}}{r_k},\\
\label{eq: freq-droop}
\dot \theta_k &= {\omega_0}-\alpha_k \left({\theta_k}-{\theta^*_{k}}\right)+\frac{-Q_k+Q_{ref,k}}{r_k}.
\end{align}
\end{subequations}
\end{proposition}
\begin{proof}
From equation \eqref{eq: input}, we use the relationship $r_k=\sqrt{u_k^\top u_k}$ and calculate  
\begin{align*}
\dot r_k&=\frac{1}{2 r_k}(u_k^\top \dot u_k+ \dot u^\top_k u_k)\\
&=\frac{1}{r_k} u_k^\top (-[\lambda(r_k, r^*_{k})I_2+J_2\omega_k(\theta_k,\theta^*_{k})]u_k- \Pi_k v_{c}), \end{align*}
which together with the definition of $P_k$ amounts to \eqref{eq: amp-droop}.

For the second statement, we define $u_k=r_k\begin{bmatrix}
 \sin(\theta_k) \cos(\theta_k)
\end{bmatrix}^\top$. Its time derivative is given by $\dot u_k=\dot r_k\begin{bmatrix}
 \sin(\theta_k) \cos(\theta_k)
\end{bmatrix}^\top+\dot \theta_k r_k J_2 \begin{bmatrix}
 \sin(\theta_k) \cos(\theta_k)
\end{bmatrix}^\top$.
We multiply from the left by $u_k^\top J^\top_2$ to obtain $$\dot\theta_k=\frac{-1}{r_k}[\sin(\theta_k) \cos(\theta_k)] J_2 \,\dot u_k.$$ 
By substituting $\dot u_k$ as given in \eqref{eq: input}, this amounts to \eqref{eq: freq-droop}.
\end{proof}
 {Around the nominal frequency and amplitude}, in addition to the active power to amplitude, or $(P-r)$ droop, the adaptive frequency in \eqref{eq: angl-fcn} introduces a droop behavior \eqref{eq: freq-droop} between the reactive power and the phase angle deviations, or $(Q-\theta)$ droop. 
\paragraph{Van-der-Pol/Liénard Oscillator}
The virtual oscillator control was initially proposed in \cite{johnson2013synchronization} based on the implementation of Van Der Pol Oscillator, or Liénard-oscillator in general. The Liénard Oscillator is the sum of a linear harmonic oscillator with an additional non-linearity implemented electrically as a nonlinear voltage-dependent current source. In the proposed oscillator, the non-linearity in ($\lambda-\omega$) oscillator is instead in the amplitude (dependent on the input magnitude) as well as the frequency (dependent on the phase angle) functions.

\paragraph{Hopf-like Oscillators}
\begin{enumerate}
\item {\em Dispatchable-VOC (d-VOC)}
The dispatchable virtual oscillator control follows by design the normal-form of Hopf bifurcation \cite{colombino2019global}. The ($\lambda-\omega$) VOC links structurally to the amplitude control in d-VOC. 
\item {\em Andronov Hopf Oscillator}
Another variation of Hopf oscillators is the Andronov Hopf Oscillator proposed in \cite{li2020advanced, lu2019grid}. This is the same as d-VOC except for an amplitude function squared in its arguments.
\end{enumerate}
The main novelty of ($\lambda-\omega$) VOC in comparison to d-VOC/Andronov Hopf is the design of an adaptive frequency function via $\omega_k(\cdot, \cdot)$ based on phase angle deviations relative to nominal. 

{\paragraph{Matching control}
The matching control can be interpreted as a virtual oscillator with a feedback-based frequency function that exploits DC measurements (see Fig. 2 in \cite{jouini2016grid}) but without an amplitude regulation function. In general, the frequency function is a degree of freedom in the design of ($\lambda-\omega$) VOC and can be incremented with a feedback term as in the matching control of the form $$\omega(v_{dc},v^*_{dc},\theta,\theta^*)=\eta (v_{dc}-v^*_{dc})-\alpha (\theta-\theta^*),$$
where $v_{dc}$ is the DC capacitor voltage  and $v^*_{dc}$ the nominal DC voltage.
}

\section{Stability analysis}
\label{sec: stability}
We transform the overall closed-loop system into the $dq$-frame rotating at $\theta_{dq}(t)= \omega^*\,t$ as follows: 
\begin{align*}
 i_{dq}&=\mathcal{R}(\omega^*\,t)i_{O},\; v_{dq}=\mathcal{R}(\omega^*\,t) v_{c},\; u_{dq}=\mathcal{R}(\omega^*\,t) u,
\end{align*}
where $\mathcal{R}(\omega^*\,t)$ is a rotation matrix.
The overall closed-loop dynamics can be written in rotating $dq-$frame as,
\begin{align}
\label{eq: pow-sys-dq}
\dot {u}_{dq}&=  - (\Lambda(r,r^*) + K(\theta, \theta^*) )\, u_{dq}- \Pi_{dq} v_{dq},\nonumber\\
C \,\dot  v_{dq} &= -Z_G\, {v}_{dq}-\mathbf{B} \,  v_{dq}+  u_{dq},\nonumber\\
L_{O}\,\dot { i}_{dq}&= -Z_{O}\, {i}_{dq} + \mathbf{B}^\top\, v_{dq},
\end{align}  
where $\Lambda(r,r^*)$ is the amplitude function in \eqref{eq: lambda-omega-voc} and $K(\theta, \theta^*)= \mathcal{I}\otimes (J_2\,\omega_k(\theta_k,\theta^*_{k}))$ with $\omega (\theta_k,\theta^*_{k})=\alpha_k(\theta_k-\theta^*_k)$. Furthermore, we define $Z_G=G\mathcal{I}+ C\mathcal{J}\omega^*$ and $Z_O=R_O\mathcal{I}+ L_O\mathcal{J}\omega^*$ as well as $\Pi_{dq}=\mathcal{I}-\frac{v^*_{dq}v^{*\top}_{dq}}{v^{*\top}_{dq}v^*_{dq}}$. 

The dynamics \eqref{eq: pow-sys-dq} satisfy the following steady state equations in the rotating frame,
\begin{align}
\label{eq: ss-eq}
0&= -(\Lambda(r^*,r^*)+ K(\theta^*, \theta^*))\, u^*_{dq}, \nonumber\\\nonumber
0&= -Z_G\, v^*_{dq}-\mathbf{B} \, i^*_{dq}+ \, u^*_{dq},\\ 
0&= -Z_{O}\, i^*_{dq}+\mathbf{B}^\top\, v^*_{dq},
\end{align}

Note that $-\left(\Lambda(r^*,r^*)+K(\theta^*, \theta^*)\right)\; u^*_{dq}=0$ for $u^*_{dq}=0$, or $r=r^*, \theta=\theta^*$. 
To see that the origin is also a solution to the steady state equations \eqref{eq: ss-eq}, note that the remaining states satisfy the following equations with a full rank matrix.
\begin{align}
\label{eq: zero-input}
{\begin{bmatrix}
 -Z_O & \mathbf{B}^\top  \\
 -\mathbf{B} & -Z_G \\ 
\end{bmatrix}} \begin{bmatrix}
 i_{dq}^*\\v_{dq}^*
\end{bmatrix}=  \begin{bmatrix}
0\\0
\end{bmatrix},  
\end{align}
Next, we make the following assumption, 
\begin{assumption}
\label{ass: ass1}
Let $\tau^*>0$ be sufficiently small. Assume that,
\begin{align}
\frac{L_{O}}{\sqrt{L_{O}^2\omega^{*2}+R_O^2}}&<\tau^*. \label{eq: tau2}    
\end{align}
\end{assumption}

Under Assumption \ref{ass: ass1}, we can avail ourselves of singular perturbation theory for the stability analysis. For this, we write the dynamical system \eqref{eq: pow-sys-dq} as two coupled subsystems $\Sigma_1$ and $\Sigma_2$. The system dynamics can be described by the following set of differential algebraic equations,
\begin{align*}
\Sigma_1: \dot {u}_{dq}&=  - (\Lambda(r,r^*) + K(\theta, \theta^*) )\, u_{dq}- \Pi_{dq} v_{dq},\\
\Sigma_2:{C} \dot v_{dq}&= - Z_G\, v_{dq}-\mathbf{B}\, i_{dq} + \, u_{dq},\\
 0 &= -Z_O\, i_{dq}+\mathbf{B}^\top v_{dq}\,,
\end{align*}
where the algebraic equation is considered as limit of the fast dynamics of the boundary-layer system with $y'=i_{dq}-i_{dq}^*$, where $i_{dq}^*$ solves the algebraic equation of $\Sigma_2$. Furthermore, the boundary-layer system reads as,
\begin{align}
\label{eq: bound-layer1}
 \dot y'&=-Z_O\, (y'+Z^{-1}_O \,\mathbf{B}^{\top} v_{dq})+\mathbf{B}^\top v_{dq}\nonumber\\&=-Z_O y'. \,
\end{align}
The origin of \eqref{eq: bound-layer1} is exponentially stable (because $-Z_O$ is Hurwitz). Thus, the reduced system is defined by, 
\begin{subequations}
\begin{align}
&\Sigma_1: \dot {u}_{dq}=  - (\Lambda(r,r^*) + K(\theta, \theta^*) )\, u_{dq}- \Pi_{dq} v_{dq},\\
&\Sigma_2: {C} \dot v_{dq}= - L v_{dq} + \, u_{dq}, 
\end{align}
\label{eq: red-sys} 
\end{subequations}
with the impedance matrix $L=Z_G+\mathbf{B}\, Z_O^{-1} \mathbf{B}^\top$.

Theorem 11.1 in \cite{K02} can be applied to show that the difference of solution trajectories between the original \eqref{eq: pow-sys-dq} and reduced system \eqref{eq: red-sys} is of the order of $\mathcal{O}\left({L_{O}}/{\sqrt{L_{O}^2\omega^{*2}+R_O^2}}\right)$ for {\em finite} time intervals.

Next, let $\mathcal{S}'$ denote the steady state set resulting from setting \eqref{eq: red-sys} to zero. In the sequel, we study almost global stability of  $\mathcal{S}'$ for any trajectory of \eqref{eq: red-sys}. We study stability with respect to sets in the sense of Theorem 2.8 in \cite{lin1996smooth}. We refer the interested reader for details therein.

Define $\xi_2=\lambda_{\min}(L+L^\top),\, \beta_1=1, \,\beta_2=\norm{L^{-1}}^2+\gamma_{\max}\norm{L^{-1}},\, \zeta=\norm{L^{-1}}$ and $\gamma_{\max}=\max\{\gamma_1,\dots,\gamma_n\}$.

Consider the following condition.
\begin{condition}
\label{cond: stability}
For a sufficiently small $\alpha^* >0$  and a constant $\xi_1>0$, assume that,
\begin{subequations}
\label{eq: cdts}
\begin{align}
\max\{\alpha_1,\dots, \alpha_n\}<\alpha^*&, \label{eq: alpha}\\
{ \Pi (\gamma_{\max}-\Pi_{dq} L^{-1})+(\gamma_{\max}-\Pi_{dq} L^{-1})^\top \Pi}&<{ -\xi_1\Pi}, \label{eq: step1}\\
\frac{C}{\sqrt{C^2\omega^{*2}+G^2}}<\frac{\xi_1 \xi_2}{\xi_1 \zeta+\beta_1\beta_2}&, \label{eq: step2}
\end{align}
\end{subequations}
\end{condition}

We state the main result in the following theorem.
\begin{theorem}
\label{thm: main-thm}
Consider the system \eqref{eq: red-sys} under Condition~\ref{cond: stability}. The steady state set $\mathcal{S}'$ is almost globally asymptotically stable for any trajectory of \eqref{eq: red-sys}.
\end{theorem}

Inspired by \cite{colombino2019global,gross2019effect},  we prove Theorem \ref{thm: main-thm} by proposing the following road map: 
\begin{itemize}
\item \underline{Step 1}:  Consider the reduced system described by,
\begin{align}
\label{eq: red-sys1}
\dot u_{dq}= -\left(\Lambda(r,r^*) + K(\theta, \theta^*)+ \Pi_{dq}\, L^{-1} \right)u_{dq}.
\end{align}
This step aims to prove the following Proposition. 
\begin{proposition}
\label{prop: al-glo-stab}
Consider the reduced system \eqref{eq: red-sys1}. Under Condition \ref{cond: stability}, the steady state set $\mathcal{S}'_u$ of \eqref{eq: red-sys1}  is almost globally asymptotically stable for any trajectory of \eqref{eq: red-sys1}. 
\end{proposition}
 \begin{itemize}
   \item Show that $\mathcal{S}'_{uu}=\mathrm{span}\{u_{dq}^*\}, u_{dq}^*\in \mathcal{S}'_u$ is exponentially stable using $\mathcal{V}_1(u_{dq})$ for any trajectory of \eqref{eq: red-sys1}.
    \item Show that the origin is an unstable equilibrium with a region of attraction $\mathcal{S}_0$ of measure zero. See later Lemma \ref{lem: zero-leb}.
   \item  For all initial conditions  $u\in \mathcal{S}'_{uu}\setminus\{0\}$, show that the set $\{u_{dq}\in\real^{2n}\,\vert\,  r=r^*\}$ is asymptotically stable, based on LaSalle invariance principle.
 \end{itemize}
 This concludes asymptotic stability of $\mathcal{S}'_u=\mathcal{S}'_{uu}\cap \{u_{dq}\in\real^{2n}\,\vert\,  r=r^*\}$ for all trajectories of \eqref{eq: red-sys1}, except those with initial condition on $\mathcal{S}_0$. For an illustration of the proof steps, see Figure \ref{fig:proof-idea}.

  \item\underline {Step 2}: We now show the main Theorem \ref{thm: main-thm}, following singular perturbation ideas from Ch.11.5 of \cite{khalil2002nonlinear}: 
\begin{itemize}
\item By treating $u_{dq}\in\mathcal{S}'_u$ as a fixed parameter, show that $v^*_{dq}=L^{-1} u^*_{dq}$ is exponentially stable using $\mathcal{V}_2(y)$ as Lyapunov function with $y=v_{dq}-L^{-1}u_{dq}$ for any trajectory of the boundary-layer system \begin{align}
\label{eq: boud-layer2}
 \dot y= -L \,(y+ L^{-1} u_{dq})+  u_{dq}=- L \,y.
\end{align}
  
\end{itemize}
  \item Use Steps 1 and 2 to show asymptotic stability of $\mathcal{S}'$ using the Lyapunov function $\mathcal{W}(u_{dq},y)=(1-d)\, \mathcal{V}_1(u_{dq})+d\,\mathcal{V}_2(y)$, with $0<d<1$ with no particular assumptions on $u_{dq}$ and $y$.
 \end{itemize}

\paragraph{Step 1}
Our goal is to investigate the stability of the subspace $\mathcal{S}'_{uu}=\mathrm{span}\{u^*_{dq}\}$ with $u^*_{dq}\in\mathcal{S}'_u$. 
In the sequel, we take $h(u_{dq})=L^{-1} u_{dq}$ and consider the reduced system \eqref{eq: red-sys1}

We have the following lemma,
\begin{lemma}
\label{lem: zero-leb}
Consider the system in \eqref{eq: red-sys1}. The set
\begin{align}
\label{eq: Y-set}
\mathcal{S}_0=\{ u_0\in\real^{2n}: \lim_{t\to\infty}u_{dq}(t, u_0)=0\},
\end{align}
has zero Lesbegue measure and the origin is unstable.
\end{lemma}
\begin{proof}
We inspect the Jacobian of \eqref{eq: red-sys1} given by,
\begin{align}
\label{eq: Jacobian}
   \left.\frac{\partial f}{\partial u_{dq}}\right.\vert_{u_{dq}=0}=\lim_{u_{dq}\to 0}\left(-\Lambda(r,r^*)-K(\theta,\theta^*)- \Pi_{dq}\, L^{-1}\right).
 \end{align} 
Note that the first two terms read, for all $k=1,\dots,n$ as,
\begin{align*}
&\lim_{u_{dq,k}\to 0}-\lambda_k(r_k,r^*_{k})I_2-\omega_k(\theta_k,\theta^*_k) J_2\\
&=\lim_{u_{dq,k}\to 0}-\gamma_k\left(\frac{\sqrt{u_{dq,k}^\top u_{dq,k}}}{r^*_{k}}-1\right)I_2-\alpha_k({\theta_k}-{\theta^*_k})J_2  \\
&=\gamma_k\cdot I_2+\alpha_k \theta_k^* J_2=:A_k.
\end{align*}
The symmetric part of $A_k$ is positive definite. Hence, for any vector $v_{dq}^{*}$,  $v_{dq}^{*\top}(\mathrm{diag}(A_1, \dots A_n)- \Pi_{dq} L^{-1})v_{dq}^{*}=v_{dq}^{*\top}\mathrm{diag}(A_1, \dots A_n)v_{dq}^{*}>0$. 
This shows that \eqref{eq: Jacobian} has at least two eigenvalues with positive real part and the origin is unstable.
By applying Prop. 11 in \cite{monzon2006local}, we conclude that the set $\mathcal{S}_0$ given in \eqref{eq: Y-set} has measure zero.
\end{proof}
\begin{figure}[h!]
    \centering
    \includegraphics[scale=0.35, trim=0cm 4cm 0 4cm, clip]{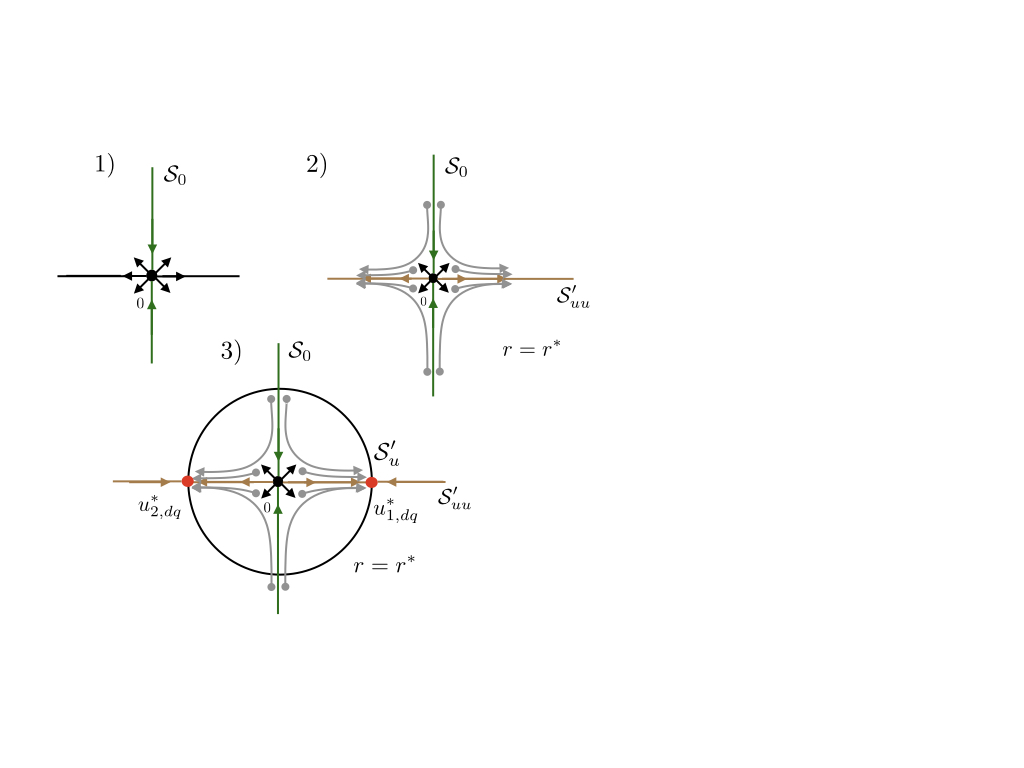}
    \caption{Sketch of Step 1 in the proof. The set $\mathcal{S}_0$ is denoted by the vertical green line. The set $\mathcal{S}'_{uu}$ is represented by the horizontal brown line and the set $\mathcal{S}_u'$ is denoted by the red dots.}
    \label{fig:proof-idea}
\end{figure}

We use as Lyapunov function an projector onto the orthogonal complement of $ u^{*}_{dq} $ as follows,
\begin{align}
{\mathcal{V}}_1(u_{dq})= u^\top_{dq}\, \Pi\, u_{dq}, ~~ \Pi= \mathcal{I}-\frac{u^*_{dq} u^{*\top}_{dq} }{u^{*\top}_{dq} u^*_{dq}}.
\end{align}
For this, we calculate,
\begin{subequations}
\begin{align*}
&\dot{\mathcal{V}}_1(u_{dq})= -2 u^\top_{dq}\,\Pi\left((\Lambda(r,r^*) + K(\theta, \theta^*) + \Pi_{dq} L^{-1} \right)u_{dq} ,\\
&\leq - 2 \, u^\top_{dq}\,\Pi\, \Pi_{dq} \, L^{-1} \,u_{dq}+ 2\, u^\top_{dq}\Pi\, \Gamma \, u_{dq}\\
&\leq 2 \, u^\top_{dq}\, \Pi\, (\Gamma-\Pi_{dq}\, L^{-1}) \, u_{dq}\leq 2 \, u^\top_{dq}\, \Pi\, (\gamma_{\max}\mathcal{I}-\Pi_{dq}\, L^{-1}) \, u_{dq},
\end{align*}
\end{subequations}
where $\Gamma=\mathrm{diag}\{\gamma_1 I_2, \dots, \gamma_n I_2 \}, \gamma_{\max}=\max\{\gamma_1, \dots,\gamma_n\}$, and under Condition \eqref{eq: alpha}, it holds that,
\begin{align*}
&-(\Pi\, (\Lambda(r, r^*)+K(\theta, \theta^*))+(\Lambda(r, r^*)+K(\theta, \theta^*))^\top\Pi)\\
&\leq \Pi\; \Gamma+\Gamma^\top \Pi. \end{align*}
The details of the calculations are in the Appendix \ref{sec: appendix}. 

{By choosing $\gamma_{\max}$, so that \eqref{eq: step1} is satisfied, we have that,
$$u_{dq}^\top \Pi (\gamma_{\max}\mathcal{I}-\Pi_{dq}\, L^{-1}) u_{dq}\leq - u_{dq}^\top \Pi \xi_1  u_{dq}.$$}
Lie derivative is hence negative definite with respect to $\mathcal{S}'_{uu}$. This shows that $\mathcal{S}'_{uu}$ is exponentially stable.

Given the amplitude and angle functions \eqref{eq: amp-fcn}, \eqref{eq: angl-fcn}, our remaining goal in this step is to show that $u_{dq}$ converges to the steady state values of the amplitudes on $\mathcal{S}'_u$. For this, we introduce the following Lyapunov function, $$\mathcal{Z}(r)= \frac{1}{2}\sum_{k=1}^n (1-\frac{r_k}{r^*_{k}})^2,$$ 
whose Lie derivative is given by, 
\begin{align*}
 \dot{\mathcal{Z}}(r)&= \sum_{k=1}^n  (1-\frac{r_k}{r^*_{k}}) \;\dot r_k\\
 &= \sum_{k=1}^n  -\gamma_k (1-\frac{r_k}{r^*_{k}})^2r_k -u_{dq}^\top C \,  \, \Pi_{dq} L^{-1} u_{dq},
\end{align*}
where $C=\mathrm{diag}\{\frac{1}{r_{1}}(1-\frac{r_1}{r^*_{1}}),\dots,\frac{1}{r_{n}}(1-\frac{r_n}{r^*_{n}}) \}$.
From the previous step, we know that $\mathcal{S}'_{uu}$ is exponentially stable. This implies that $L^{-1} u_{dq}\to L^{-1} \mathrm{span}\{u^*_{dq}\}=\mathrm{span}\{v^*_{dq}\}$ and hence $\lim_{t\to\infty}\norm{\Pi_{dq} L^{-1} u_{dq}(t)}\to 0$. From $\norm{\Pi_{dq} L^{-1} u_{dq}(t)}\leq \norm{\Pi_{dq} L^{-1} u_{dq}(0)}$ and by setting $\norm{\Pi_{dq} L^{-1} u_{dq}(0)}=0$, we have,
\begin{align*}
\dot{\mathcal{Z}}(r)\leq -\sum_{k=1}^n \gamma_k {(1-\frac{r_{k}}{r^*_k})^2}\,r_k. 
\end{align*}
Note that by $u_{dq}\neq0$, it holds that $r_k>0$. Since the sub-level sets of $\mathcal{Z}(r)$ are compact, the trajectories of \eqref{eq: red-sys} starting on $\Omega_{c_0}(u_{dq})=\{\mathcal{Z}(r)\leq c_0\}$ are bounded. We invoke LaSalle invariance principle to deduce asymptotic stability of the set $\mathcal{S}'_u$ for all trajectories starting from $\mathcal{S}'_{uu}$.

Hereby, we infer almost global stability of the desired steady state set $\mathcal{S}'_u=\{u_{dq}\in\real^{2n}\vert r=r^*\} \cap \mathcal{S}'_{uu} $ and conclude the proof of Proposition \ref{prop: al-glo-stab}.

\paragraph{Step 2}
In this step, consider the boundary-layer system given by \eqref{eq: boud-layer2}. Since $-L$ is Hurwitz, the origin is exponentially stable for all trajectories of the boundary-layer system \eqref{eq: boud-layer2} with the Lyapunov function $\mathcal{V}_2=\frac{1}{2} y^\top y$. 

Next, we consider the convex combination $$\mathcal{W}(u_{dq}, v_{dq})=(1-d){\mathcal{V}}(u_{dq})+d \mathcal{V}_2(y).$$
By taking the derivative, 
\begin{align*}
& \dot{\mathcal{W}}(u_{dq}, v_{dq})= (1-d) \dot{\mathcal{V}}(u_{dq})+ 2\, d \, y^\top \dot y \\
&=2\, (1-d)\, u^\top_{dq} \Pi (-(\Lambda(r,r^*) + K(\theta, \theta^*))u_{dq}\\ &-\Pi_{dq}\,(y+L^{-1}u_{dq}))) + 2 \frac{d}{\epsilon}\, y^\top (- L \,y)\\
&- 2\, d y^\top L^{-\top}(-\left(\Lambda(r,r^*) + K(\theta, \theta^*))u_{dq}- \Pi_{dq}\,(y+L^{-1}u_{dq}) \right).
\end{align*}

For the mixed terms in $u_{dq}$ and $y$, we calculate 
{
\begin{align*}
 &y^{\top}L^{-\top}[(\Lambda(r,r^*)+K(\theta,\theta^*)) u_{dq}+ \Pi_{dq}(y+L^{-1} u_{dq})] \leq\\  &\norm{y}(\norm{L^{-\top}\Pi_{dq}\, L^{-1}}+\gamma_{\max}\norm{L^{-\top}}) \norm{u_{dq}}_{\mathcal{S}'_u}+\norm{y}^2\norm{L^{-\top}\Pi_{dq}}\\
 &\leq \norm{y}(\norm{L^{-1}}^2+\gamma_{\max}\norm{L^{-1}}) \norm{u_{dq}}_{\mathcal{S}'_u}+\norm{y}^2\norm{L^{-1}},
 \end{align*}
 as well as the inequality $-u_{dq}^\top \Pi\, \Pi_{dq}\, y \leq \norm{u_{dq}}_{\mathcal{S}'_u} \norm{y}$.
} By condition \eqref{eq: alpha}, we only use $\norm{L^{-\top}\Lambda}\leq \norm{L^{-\top}\gamma_{\max}}$. This can be seen from $\norm{L^{-\top}\Lambda}=\norm{\Lambda\, L^{-1}}$ and $L^{-\top}\Lambda\, \Lambda\, L^{-1}\leq \gamma_{\max}^2 L^{-\top} L^{-1}$, where $L^{-1}(\Lambda\, \Lambda^\top -\gamma^2_{\max} \mathcal{I})L^{-\top}\leq 0$. Thus,
\begin{align*}
&\dot{\mathcal{W}}(u_{dq}, y)\leq\\
&\begin{bmatrix}
\norm{u_{dq}}_{\mathcal{S}'_u} \\ \norm{y} 
\end{bmatrix}^\top\left[\begin{smallmatrix} 
-2(1-d) \xi_1 &\; -\frac{1}{2}(1-d)\beta_1-\frac{1}{2}d\beta_2\\ 
-\frac{1}{2}(1-d)\beta_1-\frac{1}{2}d\beta_2 & d (\frac{\xi_2}\epsilon-\zeta)
\end{smallmatrix}\right]\begin{bmatrix}
\norm{u_{dq}}_{\mathcal{S}'_u}  \\ \norm{y} 
\end{bmatrix}.
\end{align*}
By choosing $\epsilon$  in \eqref{eq: step2} as explained in Ch.11.5 in \cite{khalil2002nonlinear}, we note that the Lie derivative of $\mathcal{W}$ is negative definite with respect to $\mathcal{S}'$.

The above arguments show that $\mathcal{S}'$ is almost globally asymptotically stable for all trajectories of \eqref{eq: red-sys} and hence concludes the proof of Theorem \ref{thm: main-thm}.

\begin{remark}
We make the following remarks in order: 
\begin{enumerate}
\item Note that the steady-state values of the voltage capacitor $v_c^*$ are not necessarily required for the operation of the ($\lambda-\omega$) VOC. These can be readily obtained from the relationship $ v_{dq}^{*}= L^{-1}\,u_{dq}^*$, where $L=Z_G+\mathbf{B}\, Z_O^{-1}\, \mathbf{B}^\top$ and we can re-write the projector as $\Pi_{dq}=\mathcal{I}-\frac{L^{-1} u^{*}_{dq}u^{*\top}_{dq}L^{-\top}}{u^{*\top}_{dq}L^{-\top} L^{-1}u^*_{dq}}$. Hence, the knowledge of only the desired current $u_{dq}^*$ and the network parameters suffices to operate the ($\lambda-\omega$) VOC.
\item Our stability analysis implies that, when a frequency function $\omega(\theta, \theta^*)$ is chosen as in \eqref{eq: angl-fcn}, the control gain $\alpha_k$ needs to be chosen to be sufficiently small as in Condition \eqref{eq: alpha}. 
Example values are given later in Section \ref{sec: sims}. Note that other variants of the frequency function can be selected, involving functions with feedback from available measurements either of the DC or AC sides.
\item The higher is the value of the line resistor $R_O$, the better  condition \eqref{eq: tau2} is satisfied. This corresponds to a small value of line inductance $L_O$. Typical values of the line inductance and line capacitance are around $L_O=7.66\cdot 10^{-6} \mathrm{H},\,C_O = 1.83 \cdot 10^{-8}\mathrm{F}, C=10^{-5}\mathrm{F}$ and their time constants correspond to $5\cdot 10^{-5}$ and $10^{-6}$, respectively, which shows a difference of at least one order of magnitude between them. This justifies Assumption \ref{ass: ass1} in practical scenarios.
\item Condition \eqref{eq: step1} and  \eqref{eq: step2} depend on the steady states of the prescribed values of the input current/capacitor voltage and in particular on the desired voltage angles. They also depend on the network parameters, encoded in the impedance matrix $L$.
\item Condition \eqref{eq: step2} is understood as a condition on the converter parameters. In particular, it verifies whether the resistive load with conductance $G>G_{\min}$ satisfies \eqref{eq: step2}, where $G_{\min}$ is the minimal acceptable load for which \eqref{eq: step2} is fulfilled.
\end{enumerate}
\end{remark}

\section{Simulations}
\label{sec: sims}
{Our simulation setup consists of identical three-phase DC/AC converter model in closed-loop with the ($\lambda-\omega$) VOC as depicted in Figure \ref{fig:setup} and connected via  transmission lines according to $\Pi$ model. The parameter values can be found in Table \ref{table_example}. 
\begin{table}[h!]
	\caption{Parameter values of the three DC/AC converters and the lines (in S.I.).}
	\begin{center}
		\begin{tabular}{|l||l|l|}
			\hline
			& $C_i,\; i=\{1,2,3\}$ & RL Lines\\
			\hline\hline
			$L$ & $5\cdot10^{-4}$ &-- \\
			$C$ & $10^{-3}$ &-- \\
			$G$ & 0.5 & -- \\ 
			$R_{}$ & 0.1 &-- \\
			$C_O$ &--  &$10^{-8}$ \\
			$R_{O}$ &--  & $0.2$ \\
			$L_{O}$ & -- & $5\cdot 10^{-5}$\\
			$\gamma_{1, 2,3}$ & 0.1 &-- \\
			$\alpha_{1, 2,3}$ & 0.03 &-- \\
			\hline
		\end{tabular}
	\end{center}
	\label{table_example}
\end{table}

We operate the converter initially at the desired steady state with currents and voltages values given by $u_k^*=r_k^*[\cos(\theta_k^*) \sin(\theta_k^*)]$ with $r_k^*=20$ and $\theta_k^*=2\pi 50 t+1.1780$ with $k=1,2,3$ and corresponding output voltage of $v_{c,k}^*=r_{c,k}^*\,[\cos(\theta_{c,k}^*) \sin(\theta_{c,k}^*)]$ with $r_{c,k}^*=175$ and $\theta_c^*=2\pi 50 t-2.463$. 
We select all control gains and system parameters satisfying the sufficient stability conditions \eqref{eq: cdts}.
We demonstrate the droop behavior of ($\lambda-\omega$) VOC as discussed in Proposition \ref{prop: droop} via an increase in power demand at converter 1. As shown in Figure \ref{fig: fig1}, notice that an increase in the conductance value (modeling an increase in active power consumption) causes a drop in the amplitude of the output capacitor voltage at the affected converter (blue line). On the other hand, an increase in the voltage amplitude of the two other converters (red and yellow lines) compensate for the total increase in power demand. This can also be seen from the power profile of the three converters in Figure \ref{fig: fig2}.

Finally, we vary the initial condition at one of the capacitor voltages $v_{c1}(0)$ and plot the electromagnetic transients of the output capacitor voltage from different initial conditions. Even though our analysis considers a reduced-order model, the trajectories of \eqref{eq: pow-sys-dq} converge to their desired amplitude and phase angles, as shown in Figure \ref{fig: fig3}.

\begin{figure}
    \centering
    \includegraphics[scale=0.48]{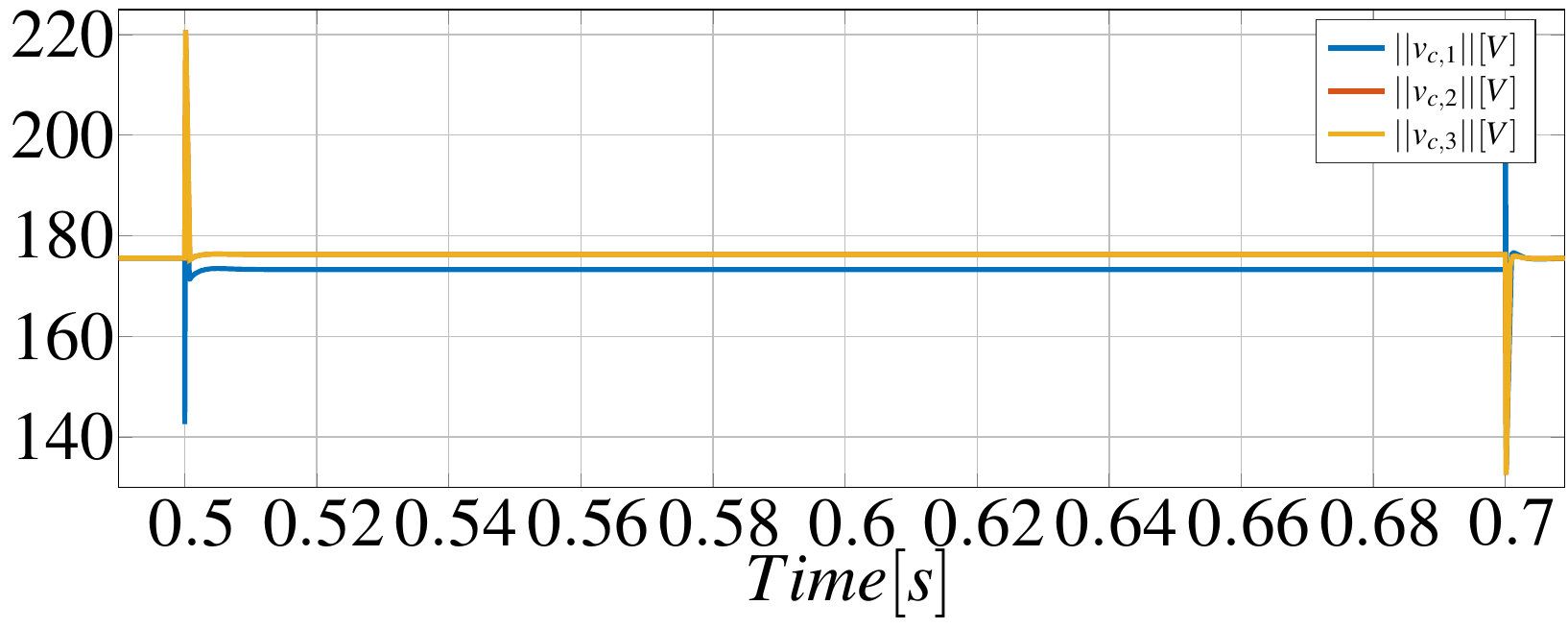}
    \caption{Droop behavior in the amplitudes of the output voltages, at converters 1,2 and 3, as converters 1 experiences an increase in power demand. These correspond to the power profiles in Figure \ref{fig: fig2} .}
    \label{fig: fig1}
\end{figure}

\begin{figure}
    \centering
    \includegraphics[scale=0.5]{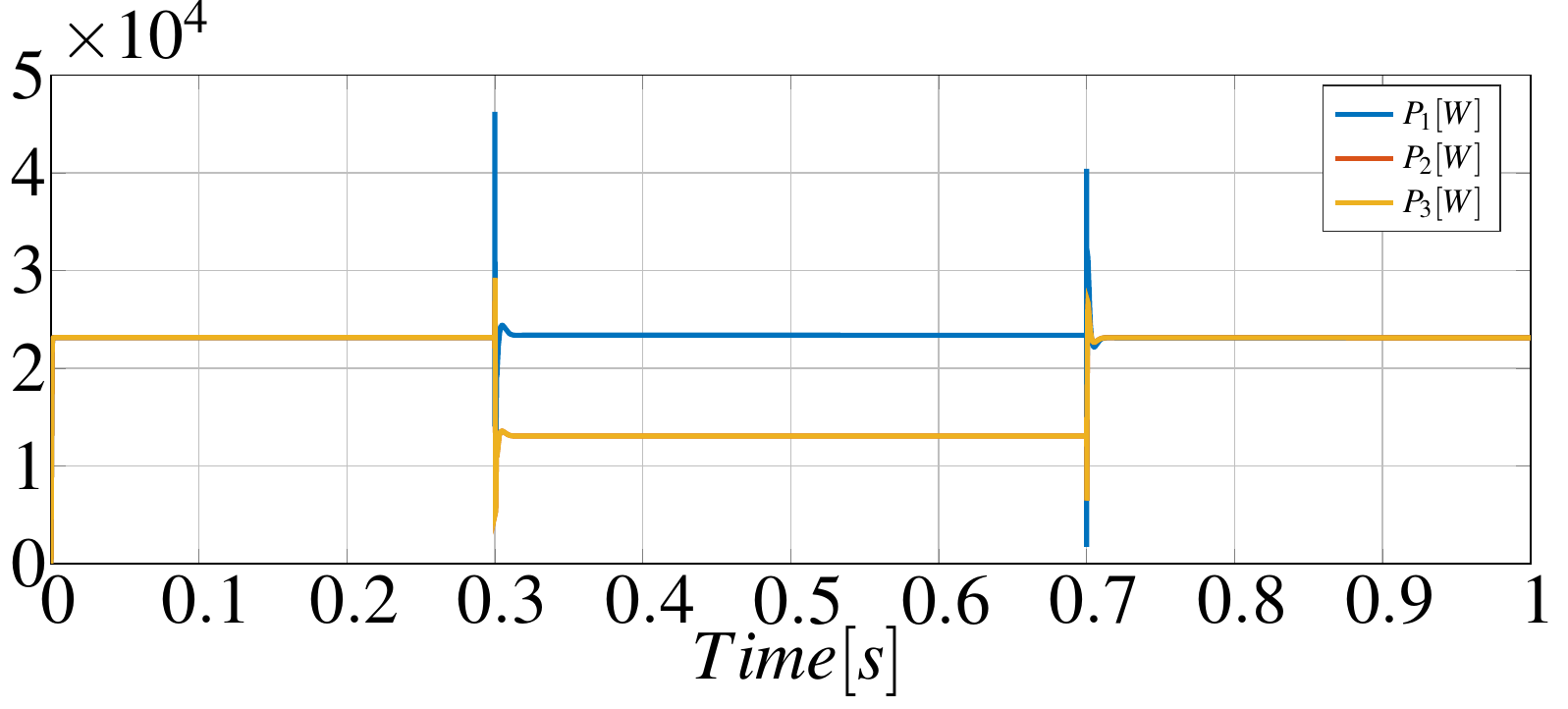}
    \caption{Active power output of the converters 1, 2 and 3 following a step in the load conductance $G$ of converter 1. Converters 2 and 3 provide power support. }
    \label{fig: fig2}
\end{figure}

\begin{figure}[h!]
    \centering
    \includegraphics[scale=0.5]{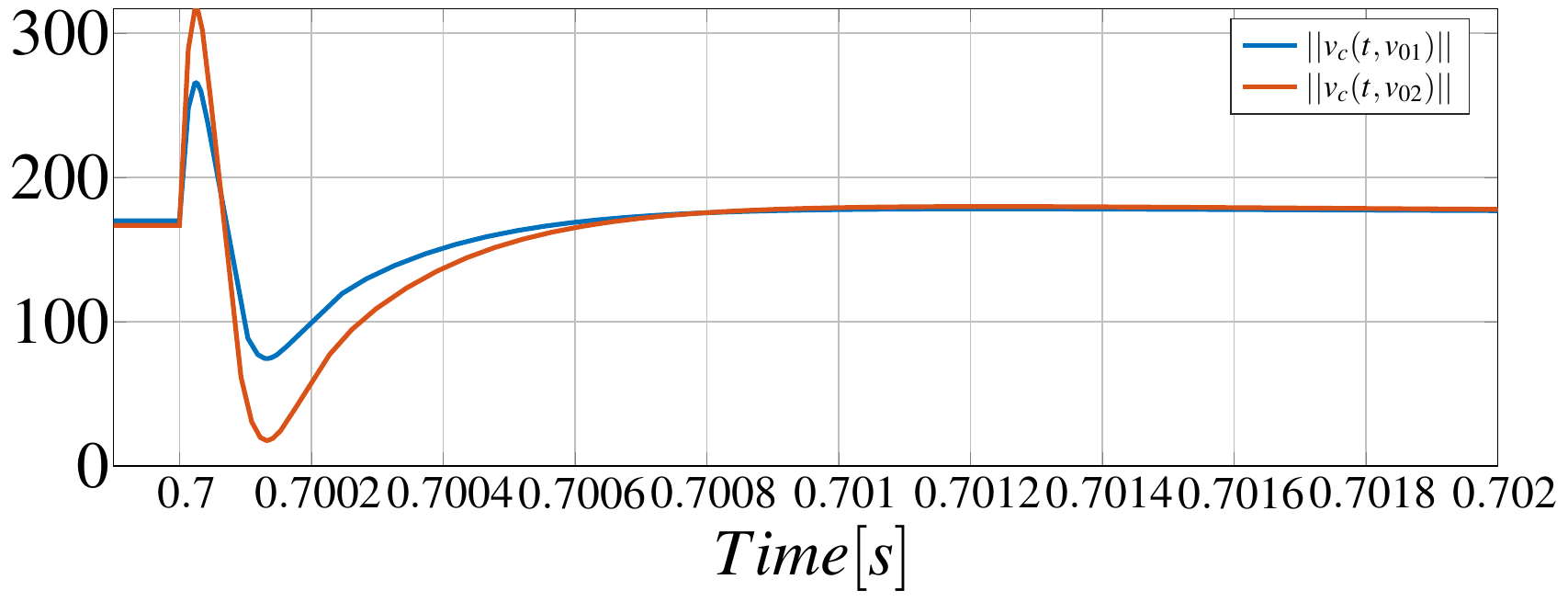}
    \caption{Electromagnetic transients of the output voltage magnitude of converter 1 with two different initial conditions (after two different disturbances).}
    \label{fig: fig3}
\end{figure}
}

\section{Conclusion}
We proposed a grid-forming control strategy inspired by wave phenomena in biological systems and building upon virtual oscillator control, while incrementing their dynamics with an adaptive frequency function. Our controller has droop behavior both in the amplitude and frequency. We proved almost global asymptotic stability of the desired set for a reduced-order model of converter network in closed-loop with ($\lambda-\omega$) VOC for a specified range of network parameters and control gains. We validated our results on a simulative three converter example. In our future work, we will consider a more detailed model of DC/AC converters and in particular include the DC-side dynamics in our analysis.
\section*{APPENDIX}
\label{sec: appendix}
\begin{proof}
Let $M_1(r)=\mathrm{diag}\{\gamma_1\,\left[\begin{smallmatrix}
\frac{r_1}{r^*_1} & 0 \\ 0 & \frac{r_1}{r^*_1}
\end{smallmatrix}\right], \dots, \gamma_n\,\left[\begin{smallmatrix}
\frac{r_n}{r^*_n} & 0 \\ 0 & \frac{r_n}{r^*_n}
\end{smallmatrix}\right]\}$ and $M_2(\theta)=\!\mathrm{diag}\{\!\alpha_1\!\left[\begin{smallmatrix}
0 & -(\theta_1-\theta_1^*) \\ (\theta_1-\theta_1^*)  & 0
\end{smallmatrix}\right]\!\!,\dots, \!\alpha_n\!\left[\begin{smallmatrix}
0 & -(\theta_n-\theta_n^*) \\ (\theta_n-\theta_n^*)  & 0
\end{smallmatrix}\right]\}$. We calculate $ $
\begin{align*}
&(\mathcal{I}-\frac{u_{dq}^*u_{dq}^{*\top}}{u_{dq}^{*\top} u_{dq}^*}) \left(M_1(r)+M_2(\theta)\right)\\
&+ \left(M_1(r) +M_2(\theta)\right)^\top(\mathcal{I}-\frac{u_{dq}^*u_{dq}^{*\top}}{u_{dq}^{*\top} u_{dq}^*})=\\
&\underbrace{M_1(r)+M_1^\top(r)-(\frac{u_{dq}^*u_{dq}^{*\top}}{u_{dq}^{*\top} u_{dq}^*} M_1(r)+M^\top_1(r)\frac{u_{dq}^*u_{dq}^{*\top}}{u_{dq}^{*\top} u_{dq}^*}}_{D_{1,k}(\gamma_k)})\\
&+\underbrace{M_2(\theta)+M_2^\top(\theta)}_{0}
-\underbrace{(\frac{u_{dq}^*u_{dq}^{*\top}}{u_{dq}^{*\top} u_{dq}^*} M_2(\theta)+M^\top_2(\theta)\frac{u_{dq}^*u_{dq}^{*\top}}{u_{dq}^{*\top} u_{dq}^*}}_{D_{2,k}}).
\end{align*} 
We calculate $D_{1,k}(\gamma_k)$ as follows,
\begin{align*}
\!\!\!\!\begin{bmatrix}
 2 \gamma_k \frac{r_k}{r^*_k}-\frac{2r_k r_k^*}{\norm{u_{dq,k}^*}^2}(\gamma_k s^2(\theta_k^*) r_k & -\frac{2r_k r_k^*}{\norm{u_{dq,k}^*}^2}\gamma_k c(\theta_k^*) s(\theta_k^*) \\-\frac{2r_k r_k^*\gamma_k}{\norm{u^*}^2} c(\theta_k^*) s(\theta_k^*) & 2 \gamma_k \frac{r_k}{r^*_k} -\frac{2r_k r_k^*}{\norm{u_{dq}^*}^2}(\gamma_k c^2(\theta_k^*) r_k 
\end{bmatrix}
\end{align*}
\begin{align*}
D_{2,k}&=(\theta_k-\theta_k^*) \alpha_k r_k^*\begin{bmatrix}
 2 s(\theta_k^*) c(\theta_k^*) & (c(\theta_k^*)-s(\theta_k^*))\\ 
 (c(\theta_k^*)-s(\theta_k^*)) & -2 s(\theta_k^*) c(\theta_k^*)
\end{bmatrix}
\end{align*}
where $s(\cdot)=\sin(\cdot),\, c(\cdot)=\cos(\cdot)$ and $u_{dq,k}^*=r_k^*[c(\theta_k^*) s(\theta_k^*)]^\top$. 
This amounts to checking $\mathrm{diag}\{D_{1,1}(\gamma_1)-D_{2,1},\dots, D_{1,n}(\gamma_n)-D_{2,n}\}$ for positive definiteness. Note that $D_1(\gamma_k)$ is positive definite (by checking the trace and determinant). Hence, for sufficiently small $\alpha_k$, $D_1(\gamma_k)-D_2$ is positive semi-definite. Since $D_1(\gamma_{\max})-D_{1,k}(\gamma_k)$ is positive definite with $\gamma_{\max}=\max(\gamma_1,\dots \gamma_n)$, it follows that $\Pi\, \Gamma +\Gamma\,\Pi\leq 2 \gamma_{\max}\Pi$.

\end{proof}




\section*{ACKNOWLEDGMENT}
Taouba Jouini would like to thank Prof. Johannes Schiffer for the insightful comments and constructive discussions.
\bibliographystyle{ieeetran} 
\bibliography{root}

\end{document}